\documentclass[12pt]{amsart}

\usepackage{amsmath}
\usepackage{amssymb}
\usepackage{amsfonts}
\usepackage{mathrsfs}
\usepackage{amsthm}
\numberwithin{equation}{section}
\usepackage[all]{xy}
\usepackage{graphicx}
\usepackage{hyperref}
\usepackage{verbatim}
\hypersetup{
    colorlinks,
    citecolor=red,
    filecolor=black,
    linkcolor=blue,
    urlcolor=black
}

\def \IZ{\mathbb Z}

\def \IT{\mathbb T}

\def \IR{\mathbb R}

\newcommand{\gre}{{\epsilon}}

\newcommand{\arw}{\rightarrow} 

\newcommand{\mb}{\mathbf}

\newcommand{\art}{\ar@{-}}

\newtheorem{thm}{Theorem}[section]
\theoremstyle{definition} \newtheorem{defin}[thm]{Definition}
\theoremstyle{plain} \newtheorem{lemma}[thm]{Lemma}
\theoremstyle{plain} \newtheorem{prop}[thm]{Proposition}
\theoremstyle{remark} \newtheorem{rem}[thm]{Remark}
\theoremstyle{plain} \newtheorem{cor}[thm]{Corollary}
\theoremstyle{remark} \newtheorem{exam}[thm]{Example}

\newcommand{\trm}[1]
{\textrm{#1}}

\newcommand{\SL}{\textrm{SL}}

\newcommand{\HR}{\textrm{\bf{H}}(\IR)}
\newcommand{\HZ}{\textrm{\bf{H}}(\IZ)}
\newcommand{\XH}{X_\textrm{\bf{H}}}
\newcommand{\Th}{\Theta^\textrm{\bf{H}}}
\newcommand{\Aut}{\textrm{Aut}}
\newcommand{\gx}{g_{\textrm{\bf{H}}}}
\newcommand{\Orb}{\mathcal{O}}
\newcommand{\quo}{\backslash}

\newcommand{\mh}{\mu_\mathbf{H}}
\newcommand{\me}{\mu_\mathbf{E}}

\newcommand{\vct}[3]
{\left(\begin{array}{c}#1 \\ #2 \\ #3\end{array}\right)}

\newcommand{\aut}[2]
{\left(\begin{array}{cc}#1 & #2 \\ 0 & 1\end{array}\right)}

\newcommand{\mat}[4]
{\left(\begin{array}{cc}#1 & #2 \\ #3 & #4\end{array}\right)}

\newcommand{\vc}[2]
{\left( \begin{array}{c} #1 \\ #2 \end{array}\right)}

\author{Jayadev S. Athreya and Ioannis Konstantoulas}
\title{Lattice deformations in the Heisenberg group}
\begin{document}
\begin{abstract}
The space of deformations of the integer Heisenberg group under the action of 
$\Aut(\HR)$ is a homogeneous space for a non-reductive group.  We analyze its 
structure as a measurable dynamical system and obtain mean and variance 
estimates for Heisenberg lattice point counting in measurable subsets of 
$\IR^3$; in particular, we obtain a random Minkowski-type theorem.  Unlike 
the Euclidean case, we show there are necessary geometric conditions on the 
sets that satisfy effective variance bounds.
\end{abstract}
\maketitle
\section{Introduction}

Minkowski's theorem in the geometry of numbers shows that in any sufficiently 
large convex centrally symmetric open set in $\IR^n$ there are non-zero 
integral points.  The asymptotic count of the number of such points is well 
understood in terms of the volume of the set, but the optimal error term is 
hard to obtain and can depend sensitively on the regularity of the boundary. 

It is an old idea that to understand a particular instance of a complicated 
system, it is beneficial to understand its typical behavior.  From the 
classical viewpoint of the `metric theory of equidistribution', $\IZ^n$ is 
one of many lattices in $\IR^n$, and their average behavior is easier to 
grasp than the individual point counting stories each lattice has to tell.   
In the more modern conception of homogeneous dynamics, $\IZ^n$ is a point in 
the space of unimodular lattices in $\IR^n$, a finite volume homogeneous 
space of $\SL(n,\IR)$.  Thus, we can formulate questions about 
the average count of lattice points and their variance for a given set, with respect to the Haar probability measure on this space.  

The first steps in that program were taken by Siegel~\cite{Sieg}  who 
considered averaged lattice point counting over unimodular lattices and gave 
a mean value formula. Subsequently, Rogers~\cite{Rog} studied higher moments of functions on the space of 
lattices and obtained a variance bound in $\IR^n$ with 
$n\geq 3$.  

Rogers's work was used by W. Schmidt in \cite{Schm} to show that for a nested 
family of Borel sets in $\IR^n$ with unbounded finite volumes, almost all 
lattices have the expected number of lattice points with explicit discrepancy 
bounds.  The hardest part in Schmidt's work was the case $n=2$ where most of 
Rogers's identities were not applicable: there was no variance bound in 
$\IR^2$ to rely upon and he had to work `by hand' using classical estimates 
from analytic number theory and the action of $\SL(2,\IZ)$ on pairs of 
integer vectors.  

The issue of a variance bound was also treated by Randol~\cite
{Rand} who obtained Rogers-type variance estimates for primitive lattice 
point counting in disks using the spectral decomposition of 
$L^2_c(\SL(2,\IR)/\SL(2,\IZ))$.  Using Randol's results and a deeper analysis 
of the Siegel operator (defined in \cite{Sieg}) for $\SL(2,\IR)$, Athreya-Margulis obtained in \cite{a_m} a variance estimate for primitive lattice 
point counting allowing arbitrary Borel sets in $\IR^2$.  

The results above provide significant information on the average behavior 
of Euclidean lattices.  A natural subsequent question that arises involves lattices in 
non-abelian groups.  The recent work of Garg, Nevo and Taylor \cite{GNT} 
addresses the lattice point count of $\IZ^{2n+1}$ in large centered balls in 
certain norms homogeneous with respect to the dilations of the $(2n+1)-$
dimensional Heisenberg group.

The present work provides variance bounds for random lattices in the case of 
the $3$-dimensional Heisenberg group.  We will see that even in this modest 
excursion outside the abelian world, the average behavior of lattices is much 
more erratic than the Euclidean case.  In particular, we show in Proposition \ref{high-disc} that very 
simple sets $S\subset \IR^3$ have no useful variance bound for the 
number of primitive points of a random Heisenberg lattice in $S$.  On the 
other hand, we also provide a natural class of sets for which optimal 
variance bounds exist and discuss extensions (see eg. Corollary \ref{boxes}).

Our main technique is to realize the space of Heisenberg lattices as a fiber bundle over the space of Euclidean lattices, use the action of the automorphism group of $\HR$ on it and relate it to the action of $\SL(2,\IR)$ on the space of Euclidean lattices.
Since the space of Heisenberg lattices embeds into the space $X_3$ of Euclidean lattices in $\IR^3$, our results can also be seen as looking at orbits in $X_3$ of certain lower dimensional subgroups of $\SL(3,\IR)$.

\subsection{Heisenberg group}The real 
Heisenberg group $\HR$ has $\IR^3$ as the underlying manifold with the smooth 
addition law (see Section \ref{EH_lat} for details) 
\begin{equation*}
\vct{r}{s}{t}+\vct{r'}{s'}{t'}=\vct{r+r'}{s+r'}{t+t'+rs'-sr'}.
\end{equation*} 
With this law $\IZ^3$ becomes a discrete subgroup denoted by $\HZ$ which will 
be taken to have co-volume $1$ with respect to the Lebesgue (Haar) measure on $\IR^3$.
Let $\XH$ be the orbit of $\IZ^3$ under the action of the connected 
component of the identity of the automorphism group of $\HR$ preserving 
volume $\textrm{Aut}^+_1(\HR)$ (this is a connected Lie group whose structure we describe in Section 
\ref{EH_lat}).  This set has the structure of a finite volume homogeneous 
space with the projected Haar measure from $\textrm{Aut}^+_1(\HR)$.  

Our work involves the following class of sets: consider a Borel set $A\subset 
\IR^2$ of finite measure greater than $1$.  The \emph{epsilon-plate} over $A$ 
at level $z$ is the set $$A^z_\gre: = (z,0,0)+A\times [0,\epsilon)\subset 
\IR^3.$$  A point $l=(m,n,k)$ in a Heisenberg lattice $g\IZ^3$ is called 
primitive if $\gcd(m,n)=1$ (this is the correct analogue of primitivity in 
$\HZ$ as we shall see in Section \ref{EH_lat}).
  
The main result provides the following average deviation bound for 
primitive lattice point count.
\begin{thm}[See Corollary \ref{cheby}]
Let $\mh$ be the projected Haar measure on the space of Heisenberg lattices $\XH$ that 
are deformations of the standard lattice $\HZ$ and $m$ the Lebesgue measure in 
$\IR^3$.  Suppose $0<\gre <1$.  We have
\begin{equation}
\mh\left(\Lambda\in \XH: \left|\# (A^z_\gre\cap \Lambda_{\textrm{prim}}) -\frac{m(A^z_\gre)}{\zeta(2)}\right|>r\sqrt{m(A^z_\gre)}\right)\leq \frac{C}{r^2}.
\end{equation}
where $C$ is an absolute constant.
\end{thm}

Combining these results with standard analytic manipulations we get 
variance bounds for sets built up from a moderate number of plates.  However, there is a gap 
between the sets for which we get bounds and those for which we prove there 
is no such bound.  This reflects the limitations of our knowledge of Euclidean 
lattice point distribution in $\IR^2$.  Despite this, regarding features of 
Heisenberg lattices that come genuinely from the action of the Heisenberg 
group, the results of sections \ref{lat_points} and \ref{distr} provide a 
comprehensive picture and one that generalizes to higher dimensional Heisenberg groups.

In subsequent work we plan to treat those higher dimensional groups combining 
an analysis of lattice point distribution of symplectic Euclidean lattices 
in $\IR^{2n}$ and the semi-direct product structure of the corresponding 
$\Aut(\textrm{\bf{H}})$.  We hope that the present paper will also serve as an 
accessible introduction to the more technical results to follow.

In the next section we provide all the relevant definitions and illustrate 
the differences between Euclidean and Heisenberg lattices.
\section{The space of Heisenberg lattices}\label{EH_lat}
In this expository section we describe the structure of the space of Heisenberg 
lattices.

\begin{defin}
The three dimensional real Heisenberg group $\HR$ is defined to be the group with 
underlying set $\IR^3$ (written as column vectors) and addition law
\begin{equation*}
\vct{r}{s}{t}+\vct{r'}{s'}{t'}=\vct{r+r'}{s+r'}{t+t'+rs'-sr'}.
\end{equation*}
The integer Heisenberg group $\HZ$ is the discrete subgroup $\HR \cap \IZ^3$.
\end{defin}
\begin{rem}
The standard symplectic form appears in the addition formula for 
$\HR$.  
Any other bilinear form would give rise to a Heisenberg group whose 
structure would be determined by the antisymmetric part of the form.  
See \cite{Aus} for the reduction and more general Heisenberg groups.
\end{rem}
\begin{prop}
Lebesgue measure in $\IR^3$ is a Haar measure for $\HR$.  The group $\HZ$ is 
a lattice in $\HR$ with a fundamental domain $[0,1)^3$.  The group of 
automorphisms of $\HR$ that preserve volume and orientation is the group 
$\trm{Aut}^+_1(\IR)=: \Aut$ consisting of matrices of the form
\begin{equation*}
\left(\begin{array}{ccc}a & b & x \\ c & d & y \\ 0 & 0 & 1\end{array}\right) \simeq \SL(2,\IR)\ltimes \IR^2
\end{equation*}
where 
\begin{equation*}
g := \left(\begin{array}{cc}a & b \\ c & d\end{array}\right)\in \SL(2,\IR)\quad\textrm{ and }\quad
\vec{v} := \left(\begin{array}{c}x \\ y \end{array}\right)\in \IR^2.
\end{equation*}
We abbreviate these elements by
\begin{equation*}
\aut{g}{\vec{v}}.
\end{equation*}
The group $\trm{Aut}^+_1(\IR)$ acts on $\HR$ by
\begin{equation}\label{action}
\aut{g}{\vec{v}}\cdot \vct{r}{s}{t} = \vc{g^*\vc{r}{s}}{t-\vec{v}^t\cdot g^*\cdot\vc{r}{s}}
\end{equation}
where $g^* = (g^{-1})^t$ is the inverse transpose.  In terms of matrix multiplication, the action is
\begin{equation*}
\aut{g}{\vec{v}}^*\vct{r}{s}{t}.
\end{equation*}
\end{prop}
\begin{proof}
All the assertions can be found in \cite[I.2 - I.3]{Aus}.  The action there 
is of the transpose of the group we have by simple matrix-column 
multiplication.  Our group acts by taking inverse transpose (landing us in 
the group used by Auslander) and performing matrix-column multiplication.
\end{proof}
\begin{rem}
The choice of group representation for $\trm{Aut}^+_1$ may seem odd; we made 
this choice because we need to take quotients on the right and it is easier 
to see the fiber bundle structure of the automorphism group from the upper 
semidirect product when we take right quotients.  Whenever no confusion can arise, we 
freely switch to the lower one and its matrix action on column vectors 
rather than passing through the inverse transpose.
\end{rem}
We turn to our main object of study:
\begin{defin}
Normalize the Haar (Lebesgue) measure $m$ of $\HR$ so that the co-volume of $\HZ$ is $1$.  
The space of $\HZ$-deformations is defined to be the orbit $\trm{Aut}^+_1(\IR)
\cdot \HZ$.  It can be identified with the quotient 
$$X_\textrm{\bf{H}}=\trm{Aut}^+_1(\IR)/\trm{Aut}^+_1(\IZ)$$ where the last 
group is the group of $\IZ$-points of the automorphism group (this is the stabilizer of $\HZ$).
\end{defin}
This is essentially the space of all Heisenberg lattices of co-volume $1$.  A 
full description of all lattices in $\HR$ is given in \cite[I.2]{Aus}.  Note 
that not all lattices in $\HR$ are isomorphic as groups to $\HZ$; those that 
are, are isomorphic through an ambient automorphism of $\HR$.  This crucial 
rigidity result, among other important facts about nilpotent groups, can be 
found in \cite{Rag}.

The next proposition shows how $\XH$ is related to the space of unimodular lattices in $\IR^2$ denoted by $X$.
\begin{prop}
The space $\XH$ has an equivariant fiber bundle structure over $X$ with compact fiber over 
$g\SL(2,\IZ)$ equal to $\IR^2/g\IZ^2$.
\end{prop}
\begin{proof}
Letting an arbitrary $g_{\textrm{\bf{H}}} = \aut{g}{x}$ act on 
$\aut{\gamma}{\delta}$ with $\gamma\in \SL(2,\IZ)$, $\delta\in \IZ^2$, we get
\begin{equation*}
 \aut{g}{x}\aut{\gamma}{\delta} = \aut{g\gamma}{g\delta+x}.
\end{equation*}
Projection onto the first factor gives the base point modulo $\SL(2,\IZ)$ and 
the pullback from that point ranges over $x+g\IZ^2$.  Equivariance follows 
from the same computation and after choosing a local trivialization at the 
identity coset of $X$ others follow by the homogeneous structure of $\XH$.  
The transition maps are given by the corresponding toral isomorphisms that 
take $\IR^2/g\IZ^2$ to $\IR^2/g'\IZ^2$.
\end{proof}
\begin{cor}\label{product}
Let $\mh$ be the unique probability measure on $\XH$ invariant under $\Aut$.  
Then $d\mh(g_{\textrm{\bf{H}}}) = d\me(g)\times d\mu_g(x)$ 
where $d\me(g)$ is the projection of Haar measure on the space of Euclidean lattices $X$ that gives measure $1$ 
to $\SL(2,\IZ)$ and $d\mu_g(x)$ the probability Haar measure on the toral 
fiber.
\end{cor}
\begin{proof}
The product measure is translation invariant by the action of $\Aut$ and 
since the base points are in $\SL(2,\IR)$, all fibers must have the same 
volume.  Fubini's theorem then gives the result.
\end{proof}
\begin{rem}
Note how $\XH\hookrightarrow X_3$ as topological spaces.  This inclusion 
identifies $\XH$ with a subset of Euclidean lattices, with the following 
caveat: when we consider the two spaces not simply as topological spaces but 
as orbit spaces with a common action on $\IR^3$, we need to modify the 
inclusion as $\XH\hookrightarrow X^*_3$ where the last space is given by the 
inverse transpose automorphism of $\SL(3,\IR)$.
\end{rem}
\section{Lattice points in measurable sets}\label{lat_points}
Let $A\subset \HR$ be measurable.  Pick a Heisenberg lattice $L$ at random from 
$\XH$ using the Haar measure.  What can we say about $A\cap L$?  It turns out 
that the answer has two parts: one involves the projection of $L$ in the 
$\IR^2$ plane orthogonal to the center of $\HR$ and the other the vertical 
distribution over each lattice point in the projection.  There are no 
`slanted lines' in the central direction for Heisenberg lattices.

The corresponding question for Euclidean lattices has a very elegant answer 
that we will use extensively.  The following result is one of the main points of \cite{a_m}:
\begin{thm}[{\cite[Theorem 2.2]{a_m}}]\label{e_lat}
Let $n\geq 2$.  There exists $C_n>0$ such that if $A\subset \IR^n$ has $m(A)>0$,
\begin{equation}
\me(\Lambda\in X_n: A\cap \Lambda = \emptyset)\leq \frac{C_n}{m(A)}.
\end{equation}
\end{thm}
Here $\me$ is the probability Haar measure on Euclidean $n$-lattices $X_n$ and $m$ is Lebesgue 
measure.  In fact, the computation in \cite[Section 4.2]{a_m} implies the following stronger statement for Euclidean lattices in $\IR^2$:
\begin{thm}\label{e_sq_div}
Let $A\subset \IR^2$ with $m(A)>0$.  Then
\begin{equation}
\me\left(\Lambda\in X_2: \left|\# (A\cap \Lambda_{\textrm{prim}}) -\frac{m(A)}{\zeta(2)}\right|>r\sqrt{m(A)}\right)\leq \frac{C}{r^2}.
\end{equation}
\end{thm}

In order to derive these theorems, the authors made extensive use of 
the \emph{Theta transform} of compactly supported functions.
\begin{defin}
Let $L=g\IZ^n$ be a Euclidean lattice in $\IR^n$.  Given a function $\phi$ in 
$L^1(\IR^n)$, the theta transform is
\begin{equation*}
\Theta \phi(L) = \sum_{\lambda \in L_{\textrm{prim}}}f(\lambda)
\end{equation*}
where $\lambda$ ranges over primitive points in $L$.  When $\phi=\chi_A$, we write $\Theta\phi =\Theta_A$.
\end{defin}
 We next give a version of the theta transform adapted to our needs.
\begin{defin}\label{prim-def}
Let $L=g\IZ^3$ be a Heisenberg lattice.  An element of $L$ is called 
primitive if $\lambda = g(k,l,m)$ with the greatest common divisor 
$\gcd(k,l)=1$.  This definition corresponds precisely to the requirement that 
there is no point $g\gamma\in g\IZ^3$ such that $g(k,l,m)=g\gamma (k','l,m')$ with $(k','l,m')\in \IZ^3$.
\end{defin}
\begin{defin}\label{theta-def}
Given a function $\phi$ in 
$L^1(\IR^3)$ let
\begin{equation*}
\Th \phi(L) = \sum_{\lambda \in L_{\textrm{prim}}}f(\lambda)
\end{equation*}
where $\lambda$ ranges over primitive points in $L$.  The operator 
$\Th:\phi\arw \Th\phi$ is called the nil-theta transform of $\phi$.  When $\phi=\chi_A$, we write $\Th \phi =\Th_A$.
\end{defin}
For characteristic functions of Borel sets, the Theta transform has the 
following properties that we will use repeatedly:

\begin{prop}\label{comp_supp}
Let $\phi=\chi_A$ with $A$ a bounded Borel set such that there exists a 
neighborhood $U$ of $\IR^2$ containing $0$ such that $A\cap (U\times \IR) = 
\emptyset$.  Then the Theta 
transform $\Th\phi$ is a bounded Borel function with compact support on $\XH$.
\end{prop}
\begin{proof}
The Borel property is clear.  For a Heisenberg lattice $L$ to intersect $A$, its projection in 
$\IR^2$ must intersect $A_0=\pi_{\textrm{flat}}(A)$, i.e. we must have
$\Theta\chi_{A_0}\neq 0$.  $A_0$ being a bounded set that does not meet a 
neighborhood of the origin shows (\cite[Chapter XIII, Par.1]{Lang}) that $\Theta\chi_{A_0}$ has compact support $C$ in $X$.

Therefore the closure of the set of Heisenberg lattices giving a non-zero $\Th_A$ is a subset of 
a torus bundle over the compact set $C$ and therefore is compact.

Finally, whenever $\phi\leq \psi$ we have $\Th\phi\leq \Th\psi$ so if $\psi$ 
is a continuous function whose support is compact and contains $A$, we have 
$\Th_A\leq \Th\psi$.  Since $\Th\psi$ is continuous (continuity is proven as in loc. cit.) with compact support, it 
is bounded, and therefore so is $\Th_A$.
\end{proof}
The nil-theta transform of characteristic functions counts lattice points in 
sets like its Euclidean counterpart.  However, a result as uniform as theorem 
\ref{e_lat} cannot hold in the Heisenberg setting:
\begin{prop}
Let $\pi_{\textrm{flat}}$ be the projection in $\HR$ to the first two 
coordinates.  Let $A\subset \HR$ have positive measure.  The following 
inequality holds:
\begin{equation}
\mh(L\in \XH:A\cap L = \emptyset) \geq \me(\Lambda \in X:\pi_{\textrm{flat}}(A)\cap \Lambda = \emptyset).
\end{equation}
\end{prop}
\begin{proof}
The action \eqref{action} transforms the flat part of the lattice by an 
element of $\SL(2,\IR)$ and then translates the third component of each 
lattice point accordingly along the central direction.  Suppose the 
projection of $A$ does not intersect a lattice $\Lambda$.  Then $A$ cannot 
intersect any lift of $\Lambda$, since lifts are determined by values of the 
third coordinate over the flat part.  Thus the entire torus of lattices in 
$\XH$ over $\Lambda$ misses $A$, giving the inequality.
\end{proof}
In particular, we can increase the measure of $A$ indefinitely keeping 
$\pi_{\textrm{flat}}(A)$ fixed; in the extreme cases of Theorem \ref{e_lat} (which are attained), for some $C>0$ we have
\begin{equation*}
\mh(\Lambda \in X:\pi_{\textrm{flat}}(A)\cap \Lambda = \emptyset)\sim \frac{C}{m(\pi_{\textrm{flat}}(A))}.
\end{equation*}
We see that even if $m(A)$ itself becomes large, the probability of missing a 
Heisenberg lattice remains bounded below by the inverse of the measure of the projection.

The origin of this discrepancy is in principle easy to understand: the space 
of Heisenberg lattices is a very thin subset of the space of Euclidean $3$
-lattices; in particular, all members of $\XH$ project to $2$-lattices in the 
flat plane, so they never tilt along the central direction.  This phenomenon 
can be illustrated by the following extreme example:
\begin{exam}\label{bad_cyl}
For $\delta >0$ and $N\geq 1$ large, define $T(\delta,N)$ to be the cylindrical 
punctured tube $$T(\delta,N) = \left(D(0,\delta)\setminus (0,0)\right)\times 
[-N,N]$$ in $\IR^3$.  Consider a compact subset 
$C\subset X$ of measure $\me(C)=1-\gre$.  Then by the compactness criterion on 
the space of lattices, all lattices in $C$ have vectors of length at least 
$2\delta = \sqrt{\gre}$.  Then 
\begin{eqnarray*}
& &\mh(L \in \XH: L \cap T(\delta,N ) =\emptyset )\\ &\geq & \me(\Lambda \in X: \Lambda \cap D(0,\delta) =\emptyset)\\ &\geq & \me(C) = 1-\gre.
\end{eqnarray*}
On the other hand, the measure of $T(\delta,N)$ can become arbitrarily large 
no matter how small $\delta$ is, by increasing $N$.  
\end{exam}
This is in stark contrast to the situation for Euclidean $3$-lattices: starting from $\IZ^3$, 
we can reach the tube with an infinitesimally small shear; the points 
$(0,0,n)\in \IZ^3$ will immediately tilt to intersect $T(\delta, N)$.  

The following question arises from this discussion: given a set $A$ in $\HR$ 
and knowledge of the statistics of 
lattice points on $\pi_{\textrm{flat}}(A)$, how can we deduce the statistics of 
Heisenberg lattice points on $A$?  The next two sections are devoted to the answer.

\section{Level distribution of lattices}\label{distr}
\begin{defin}
Let $A\subset \IR^2$ be measurable with positive measure, $\gre>0$ and 
$z:=(0,0,z)$ with $z\in \IR$.  Let $$A^z_\gre = z+A\times [0,\gre)$$ be the 
level-$z$ $\gre$-plate over $A$.  Consider a Heisenberg lattice $L=\gx\HZ$.
  The set $$\{\#(L \cap A^z_\gre):\gre\in (0,1)\}$$ is called the $z$-level 
distribution of $L\in \XH$.  From now on we assume $0<\gre<1$.
\end{defin}
\begin{rem}
The definition above captures significant information for
characterizing lattice statistics given the statistics of the projection, 
because lattices above a given Euclidean lattice 
only differ in the vertical deviation of lattice points from the integer 
lattice.  As we will see, the level distribution is determined modulo $\IZ$ 
so we can take $z\in [0,1)$.
\end{rem}
\begin{prop}[{\cite[Proof of Theorem 2.2]{a_m}}]\label{e_disc}
For $A\subset \IR^2$ of positive measure $m(A)=a$, we have the variance estimate
\begin{equation}
\left\|\Theta_A-\frac{a}{\zeta(2)}\right\|_2^2\leq 16 a.
\end{equation}
\end{prop}

We now begin the study of moments of $\Th_{A^z_\gre}$.  
The following proposition gives the Siegel formula.
Since some intermediate formulas will be used in the $L^2$ estimate, we give a quick proof.
\begin{prop}\label{siegel_h}
\begin{equation}
\int_{\XH}|\Th_{A^z_\gre}(L)|\,d\mh(L) = \frac{m(A)\gre}{\zeta(2)}.
\end{equation}
\end{prop}
\begin{proof}
We have $\chi_{A^z_\gre} = \chi_A \chi_{z+[0,\gre)}$.  If $\gx=\aut{g}{\vec{x}}$, write
\begin{equation*}
\Th_{A^z_\gre}(\gx\HZ_{\trm{prim}}) = \sum_{(\vec{m},k)\in \HZ_{\trm{prim}}}\chi_A(g^*\vec{m})\chi_{z+[0,\gre)}(k-\vec{x}^t\cdot g^*\cdot \vec{m})
\end{equation*}
which splits as
\begin{eqnarray}
\Th_{A^z_\gre}(\gx\IZ^3) &=& \sum_{\vec{m}\in \IZ^2_{\trm{prim}}}\chi_A(g^*\vec{m})\sum_{k\in \IZ}\chi_{z+[0,\gre)}(k-\vec{x}^t\cdot g^*\cdot \vec{m})\nonumber\\
&=& \sum_{\vec{m}\in \IZ^2_{\trm{prim}}}\chi_A(g^*\vec{m})\sum_{k\in \IZ}\chi_{[0,\gre)}(k-z-\vec{x}^t\cdot g^*\cdot \vec{m})\nonumber\\
&=& \sum_{\vec{m}\in \IZ^2_{\trm{prim}}}\chi_A(g^*\vec{m})\mb{1}(\{z+\vec{x}^t\cdot g^*\cdot \vec{m}\}<\gre).\label{shorter}
\end{eqnarray}
We want to integrate the last expression over the fiber first, then over the base, using Corollary \ref{product}.
\begin{eqnarray}
& & \int_{\XH}\Th_{A^z_\gre}(\gx\HZ_{\trm{prim}})\,d\mh(\gx)\nonumber \\&=& \int_X\sum_{\vec{m}\in \IZ^2_{\trm{prim}}}\chi_A(g^*\vec{m})\left(\int_{\IR^2/g\IZ^2} \mb{1}(\{z+\vec{x}^t\cdot g^*\cdot \vec{m}\}<\gre )\,dx\right)\,d\me(g).\label{fiber_int}
\end{eqnarray}
Now perform the change of variables $\vec{x}=g\vec{u}$ in the inner integral and use $\det{g}=1$ to simplify to
\begin{equation*}
\int_{\IR^2/\IZ^2} \mb{1}(\{z+\vec{u}^t\cdot \vec{m}\}<\gre )\,du.
\end{equation*}
Since $\SL(2,\IZ)$ is transitive on primitive vectors, all the toral integrals must have the 
same value which we now compute.

The map $S(u_1)=z+ u_1m_1+u_2m_2 \pmod{g\IZ^2}$ is measure preserving, so splitting the integral over the torus in \eqref{fiber_int}, we get
\begin{equation*}
\int_{\IR/\IZ}\int_{\IR/\IZ} \mb{1}(S^{-1}( [0,\gre))\,du_1\,du_2 = \gre.
\end{equation*}

The outer integral in \eqref{fiber_int} is
\begin{equation*}
\int_X \Theta_A(\Lambda)d\me(\Lambda) = \frac{m(A)}{\zeta(2)}
\end{equation*}
using the Siegel formula for Euclidean lattices and our 
normalization of $X$.  
\end{proof}

Now we need some preparations for dealing with the second moment of $\Th_{A^z_\gre}$.
\begin{defin}
Fix an integer $D$.  Consider the space of pairs of coprime integer vectors 
\begin{equation} M_D=\{\mb{m} = (\vec{m},\vec{n})\in \IZ^2_{\trm{prim}}\times\IZ^2_{\trm{prim}}: 
\det{\mb{m}}=D\}.
\end{equation}
The diagonal action 
of $\SL(2,\IZ)$, $\gamma(\vec{m},\vec{n}) = (\gamma\vec{m},\gamma\vec{n})$ 
leaves invariant the determinant of $\mb{m}$, so it stabilizes each $M_D$.   
The space of $\SL(2,\IZ)$ orbits of the set $M_D\subset 
\IZ^2_{\trm{prim}}\times\IZ^2_{\trm{prim}}$ of determinant $D$ pairs is 
denoted by $M_D^{\Orb} = \SL(2,\IZ)\quo M_D$. 
\end{defin}
\begin{defin}\label{cor_def}
Let $\vec{m}=(m_1,m_2)\in \IZ^2$ and $z\in \IR$.  Define the map $S_{\vec{m}}:\IT^2\arw \IT$ by 
\begin{equation*}
S_{\vec{m}}(u_1,u_2) = z+m_1u_1+m_2u_2\pmod{1}.
\end{equation*}
Denote the pullback of $[0,\gre)$ by $S$ into $\IT^2$ by
\begin{equation*} 
C_{\vec{m}}(\gre) := S_{\vec{m}}^{-1}([0,\gre])\subset \IT^2.
\end{equation*}
Finally, for $\vec{m}$,$\vec{n}$ in $\IZ^2_{\trm{prim}}$, define the correlation of 
\begin{equation*}
Cor_{\vec{m},\vec{n}}(\gre,z)=\mu_{\IR^2/\IZ^2}(C_{\vec{m}}(\gre)\cap C_{\vec{n}}(\gre))
\end{equation*}
where $\mu_{\IR^2/\IZ^2}$ is the probability Haar measure on the $2$-torus.
\end{defin}

Next, we turn to the second moment of $\Th_{A^z_\gre}$.
\begin{prop}
We have the following identity:
\begin{equation}\label{second_moment}
\int_{\XH}|\Th_{A^z_\gre}(L)|^2\,d\mh(L) 
= \int_X \sum_{\vec{m},\vec{n}\in \IZ^2_{\trm{prim}}}\chi_A(g^*\vec{m})\chi_A(g^*\vec{n})Cor_{\vec{m},\vec{n}}(\gre,z)\,d\me(g).
\end{equation}
\end{prop}
\begin{proof}
Using notation as in Proposition \ref{siegel_h} and the expression \eqref
{shorter} for the integrand, expand the square:
\begin{eqnarray*}
|\Th_{A^z_\gre}(L)|^2 = \sum_{\vec{m},\,\vec{n}}\chi_A(g^*\vec{m})
\chi_A(g^*\vec{n})& &\\ & & \cdot \mb{1}(\{z+\vec{x}^t g^* \vec{m}\}<\gre )\mb
{1}(\{z+\vec{x}^t g^* \vec{n}\} <\gre).
\end{eqnarray*}
Let $$w_{\vec{m},\vec{n}}(g,\vec{x})= \mb{1}(\{z+\vec{x}^t g^* \vec{m}\}<\gre )\mb
{1}(\{z+\vec{x}^t g^* \vec{n}\} <\gre).$$ Once again integrating and changing variables, we get
\begin{equation*}
\int_{\IR^2/g\IZ^2}w_{\vec{m},\vec{n}}(g,\vec{x})dx = \int_{\IR^2/\IZ^2}\mb{1}(\{z+\vec{u}^t \vec{m}\}<\gre )\mb
{1}(\{z+\vec{u}^t\vec{n}\} <\gre)\,du.
\end{equation*}
Since this expression is invariant under $g\arw g\gamma$, the result follows from the definition of $C_{\vec{m}}(\gre)$.
\end{proof}
\begin{rem}
The correlation $Cor_{\vec{m},\vec{n}}(\gre,z)$ is a `correction factor' that 
weighs the double sum in \eqref{second_moment}.
Therefore, the study now reduces to the distribution of the values of 
$Cor_{\vec{m},\vec{n}}(\gre,z)$.  
Whenever $\mb{m} = \gamma\mb{n}$ are in the same $\SL(2,\IZ)$-orbit, the corresponding correlations must be the same.  
Thus one can write $Cor_{\vec{m},\vec{n}}(\gre,z) = Cor_\Orb(\gre,z)$ for the common value of the correlation over an orbit $\Orb \in M_D^\Orb$.  
The idea is to look for a simple pair $\mb{m}$ in the orbit where the 
correlation is easy to compute.
\end{rem}

The next lemma describes the structure of the above set; 
the proof can be found in \cite[Lemma 6]{Schm} or worked out by hand.
\begin{lemma}\label{orbit_str}
The structure of $M_D^\Orb$ is as follows:
\begin{enumerate}
\item If $D=0$, $M_D^\Orb$ has two elements represented by $\pm \vc{1}{0}$.
\item If $D=\pm 1$, $M_D^\Orb$ is a singleton represented by the identity matrix ($\trm{diag}(-1,1)$ respectively).
\item If $|D|>1$, then $M_D^\Orb$ consists of $\phi(D)$ distinct orbits, 
represented by $\mat{D}{0}{k}{1}$ where $k$ runs through a complete set of 
residues modulo $D$.
\end{enumerate}
\end{lemma}

Using this lemma, we can show:
\begin{prop}\label{orbits}
The correlation $Cor_\Orb(\gre,z)$ is equal to $\gre^2$ when $\Orb \in M_D$ 
with $D\neq 0$ and equal to $\gre$ when $D=0$ and $\Orb=+1$, $0$ when $\Orb=-1$.
\end{prop}
\begin{proof}
When $D=0$ with equal parity, choose $\mb{m} = \mat{1}{0}{1}{0}$; then $$Cor_\Orb(\gre,z) = 
S^{-1}([0,\gre))=\gre$$ as in Proposition \ref{siegel_h}.  For 
opposite parity, the two intervals $[0,\gre)+\IZ$ and $(-\gre,0] +\IZ$ are disjoint so their pullback under $S$ 
has empty intersection.

When $D\neq 0$, for any $\mat{D}{0}{k}{1}$ compute
\begin{eqnarray*}
& & \int_{\IR^2/\IZ^2}\mb{1}(\{z+\vec{u}^t \vec{m}\}<\gre )\mb
{1}(\{z+\vec{u}^t\vec{n}\} <\gre)\,du\\ &=& \int_{\IR/\IZ}\int_{\IR/\IZ}\mb{1}(\{z+Du_1+ku_2\}<\gre )\mb
{1}(\{z+u_2\} <\gre)\,du_1\,du_2 \\
&=& \int_{\IR/\IZ}{1}(\{z+u_2\} <\gre)\int_{\IR/\IZ}\mb{1}(\{z+Du_1+ku_2\}<\gre )\,du_1\,du_2;
\end{eqnarray*}
 the inner integrand is the pullback of $[0,\gre)$ by a (measure preserving) affine map of the 
form $Du_1+\tau$ giving integral $\gre$ independent of $u_2$, and then 
the remaining factor contributes another $\gre$.
\end{proof}
We now come to the main result of this section, the analog of Theorem \ref
{e_disc} for Heisenberg lattices.
\begin{thm}\label{centr}
Let $A^z_\gre$ be an epsilon-plate over $A$ of measure $m(A)>1$.  We have
\begin{equation}\label{l2}
\left\|\Th_{A^z_\gre} - \frac{\gre\,m(A)}{\zeta(2)}\right\|_2^2 \leq \frac{\gre\,m(A)}{\zeta(2)} + 20\gre^2 m(A).
\end{equation}
\end{thm}
\begin{proof}
Write 
\begin{equation*}
\left\|\Th_{A^z_\gre} - \frac{\gre\,m(A)}{\zeta(2)}\right\|_2^2 = \left\|\Th_{A^z_\gre}\right\|_2^2  - \left(\frac{\gre\,m(A)}{\zeta(2)}\right)^2
\end{equation*}
and use \eqref{second_moment} to expand
\begin{equation*}
\left\|\Th_{A^z_\gre}\right\|_2^2 = \int_X \sum_{\vec{m},\vec{n}\in \IZ^2_{\trm{prim}}}\chi_A(g^*\vec{m})\chi_A(g^*\vec{n})Cor_{\vec{m},\vec{n}}(\gre,z)\,d\mh(g);
\end{equation*}
by Proposition \ref{orbits} the integrand on the right hand side becomes
\begin{equation*}
\gre\sum_{\mb{m}\in M_0,\,+}\chi_A(g^*\vec{m})\chi_A(g^*\vec{n}) + \gre^2\sum_{\mb{m}\in M_D,\,D\neq 0}\chi_A(g^*\vec{m})\chi_A(g^*\vec{n});
\end{equation*}
now add and subtract the sum over $M_0$ weighted by $\gre^2$ to get
\begin{equation*}
(\gre-\gre^2)\sum_{\mb{m}\in M_0,\,+}\chi_A(g^*\vec{m})\chi_A(g^*\vec{m}) + \gre^2\sum_{\mb{m}\in M}\chi_A(g^*\vec{m})\chi_A(g^*\vec{n})
\end{equation*}
for the integrand; recall $\det{\mb{m}}=0$ and parity preservation implies 
$\vec{m}=\vec{n}$.  Integrating the two parts separately, we get
\begin{equation*}
\left\|\Th_{A^z_\gre}\right\|_2^2 = (\gre-\gre^2)\int_X \Theta\chi_A^2(L)\,d\me(L) + \gre^2 \int_X |\Theta_A(L)|^2\,d\me(L)
\end{equation*}
which simplifies to
\begin{eqnarray*}
\left\|\Th_{A^z_\gre}\right\|_2^2 &=& (\gre-\gre^2)\int_X \Theta\chi_A(L)\,d\me(L) + \gre^2 \|\Theta_A\|_2^2\\
&=& \frac{(\gre-\gre^2)m(A)}{\zeta(2)} + \gre^2\|\Theta_A\|_2^2
\end{eqnarray*}
using $\chi_A=\chi_A^2$ and the Siegel formula for Euclidean $2$-lattices 
(recall our normalization of $X$).  Now we subtract the constant term and absorb it into the second summand to get
\begin{equation}\label{raw}
\left\|\Th_{A^z_\gre}- \frac{\gre\,m(A)}{\zeta(2)}\right\|_2^2 = \frac{(\gre-\gre^2)m(A)}{\zeta(2)} + \gre^2 \|\Theta_A-\frac{m(A)}{\zeta(2)}\|_2^2.
\end{equation}
Applying Theorem \ref{e_disc} and gathering $\gre^2$ terms together, we obtain
\begin{equation}\label{disc}
\left\|\Th_{A^z_\gre}- \frac{\gre\,m(A)}{\zeta(2)}\right\|_2^2 \leq \frac{\gre\,m(A)}{\zeta(2)} + 20\gre^2 m(A).
\end{equation}
\end{proof}
Theorem \ref{centr} immediately implies the analogues of Theorems \ref{e_lat} and \ref{e_sq_div} for the plate distribution.
\begin{cor}
We have the following bound on the probability of a Heisenberg lattice missing a plate $A^z_\gre$:
\begin{equation}
\mh(\Lambda\in \XH: A^z_\gre \cap \Lambda = \emptyset)\leq \frac{C}{m(A^z_\gre)}.
\end{equation}
\end{cor}
\begin{cor}\label{cheby}
The average discrepancy of lattice points in a plate $A^z_\gre$ satisfies the Chebyshev inequality
\begin{equation}
\mh\left(\Lambda\in \XH: \left|\# (A^z_\gre\cap \Lambda_{\textrm{prim}}) -\frac{m(A^z_\gre)}{\zeta(2)}\right|>r\sqrt{m(A^z_\gre)}\right)\leq \frac{C}{r^2}.
\end{equation}
\end{cor}
From this information, we can obtain discrepancy estimates from sets built 
out of plates in a controlled number of steps.  We illustrate this with the example of a stout cylinder:

\begin{cor}\label{boxes}
Let $C = A\times I$ with $m(A)>1$ and $|I| \leq m(A)^{\frac{1}{2}-\delta}$.  Then 
\begin{equation}
\mh\left(\Lambda\in \XH: \left|\# (C\cap \Lambda_{\textrm{prim}}) -\frac{m(C)}{\zeta(2)}\right|>rm(C)^{1-\delta}\right)\leq \frac{C}{r^2}.
\end{equation}
\end{cor}
\begin{proof}
Split $C$ into at most $m(A)^{\frac{1}{2}-\delta}$ plates $P_i$ of height at most 
$1$; since these sets are disjoint, we can split $$\Th_{S} - 
\frac{m(S)}{\zeta(2)}  = \sum_{P_i}\left(\Th_{P_i} - 
\frac{m(P_i)}{\zeta(2)}\right);$$ 
then apply the Cauchy-Schwartz inequality to \eqref{l2} after taking square roots to get
\begin{equation*}
\left\|\Th_{S} - \frac{m(S)}{\zeta(2)}\right\|_2 \leq 10 m(S)^{1-\delta}.
\end{equation*}
  Then proceed as in the proof of Chebychev's inequality with the appropriate exponents.
\end{proof}
\begin{rem}
The result above is presumably not optimal; we expect that good discrepancy 
estimates will hold for approximately cubical cylinders.  
However, this 
requires higher moment analogues of \ref{e_disc} (which are not known) and a corresponding 
treatment of higher correlations in \ref{cor_def}.  
In any case, these sets 
come close to the limit of sets we should expect to satisfy good deviation 
estimates as the next section demonstrates.
\end{rem}
\section{High discrepancy sets}
Here we construct a wide variety of sets for which an estimate like the one 
in Corollary \ref{cheby} cannot hold; the argument is an extension of Example \ref
{bad_cyl}.  The point of the following proposition is to show that firstly high discrepancy sets 
need not be anchored to any specific point like the origin, and secondly do 
not need to have an approximately cylindrical shape (although they will be built out of cylinders).
\begin{prop}\label{high-disc}
Let $B(0,R)$ be a large ball in $\IR^2$ centered at the origin and fix $\gre>0$.  Let $$Z_\circ = (\IZ^2 + B(0,\gre))\cap B(0,R)$$ be a thickening of the 
standard lattice inside the fixed ball. There exist Borel 
sets $S\subset \IR^3$ such that: 
\begin{enumerate}
\item $S$ has arbitrarily large measure, 
\item for any cylinder $C$ and $\gre'>0$, $$m(C\triangle S)>_{\gre'} m(S)^{1-\gre'},$$
\item $\pi_{\textrm{flat}}(S)=A$ is an arbitrary Borel set of positive measure in $B(0,R)\setminus Z_\circ$ and
\item the inequality $$\mh(\Lambda\in \XH: S \cap 
\Lambda = \emptyset)\geq \delta $$ holds, where $\delta$ depends only on $\gre$.
\end{enumerate}
\end{prop}
\begin{proof}
  Let $U$ be a neighborhood of the 
identity in $\SL(2,\IR)$ such that $$(\overline{U}\IZ^2)\cap B(0,R)\subset Z_\circ$$ and $\me(U)=\delta$ its measure.  

Generally to meet requirements $1$, $3$ and $4$ we can pick an arbitrary Borel $A\subset B(0,R)\setminus Z_\circ$ of positive 
measure, partition it into finitely or countably many disjoint Borel sets $A_i$ of 
positive measure and over each set erect a cylinder $C_i=A_i\times I_i$ of 
any desired height to form $S$ so that $$\sum_{i} |I_i|m(A_i)=m(S)$$ is 
arbitrarily large and distributed in an arbitrary way over the bases $A_i$.

To meet the second requirement, observe that all that matters in the 
computation is the measures of the $A_i$ and the placements of the cylinders 
over them.  Thus we can visualize the $A_i$ as disjoint horizontal intervals 
on the real line arranged in decreasing size and the $I_i$ as vertical intervals at prescribed level and 
of prescribed height.

Choose a large finite partition of $A$ into $k=m(A)^N$ parts so that each $A_i$ has measure 
approximately $\frac{1}{2^i} m(A)$ and choose segments $I_1=[1,2]$, 
$I_2=[2,4]$, $I_3=[4,8]$ and so 
on.  Then $m(S) \sim km(A)$ and by controlling $k$ we can make $m(S)$ as large 
as possible.  Note that the top of the cylinder over $A_i$ is at $t_i\sim 2^{i+1}$.

A cylinder that will minimize the difference 
$C\triangle S$ will necessarily have base inside $A$, so we may assume that 
$C=(\bigcup_{j\in J}A_j)\times I$ where the union is over some sub-collection 
of the partition.  Assume $C$ contains precisely $l$ of the $k$ cylinders, with $l\gg 1$.  

Then $C$ is 
based on at least $l$ and at most $(l+2)$ of the $A_i$.  Let $r+1$ be the smallest index so that 
the cylinder over $A_{r+1}$ is contained in $C$.  In order to minimize the 
difference, the optimal cylinder has to be based on all $A_j$ for 
$j=r+1,\cdots,r+l$ (for each $j>r+1$ skipped, the gains from removing $C_j$ 
are offset by the fact that indices are shifted at least one place, doubling 
the cost of the cylinder over $A_r$; this beats the gain because the measure 
of $A_r$ is at least twice the measure of $A_j$).

For this configuration, the height of the cylinder is 
$t_{r+l}-t_r=2^{r+1}-2^r$ and the base has measure $m(A)2^{-r}(1-2^{-l})$; 
thus its measure is $m(A)2^l$.  In contrast, the measure of $S\cap C$ is 
within $(l\pm 2)m(A)$.  Therefore, $$m(S\triangle C)\geq m(A)(2^l-l-2).$$

If $l\geq \log_2(k)$ we get our result.  If $l\leq \log_2(k)$, then $$m(S\triangle C) \geq m(S)-m(C\cap S) \geq m(A)(k-l)$$

which again gives the result.

Finally, by the discussion preceding Example \ref{bad_cyl} we know that $$\mh(\Lambda\in \XH: S \cap 
\Lambda = \emptyset)\geq \me(\Lambda\in X: A \cap 
\Lambda = \emptyset)\geq \me(U)=\delta .$$ Since our choice of $U$ depended only on $\gre$, we are done.
\end{proof}


\begin{thebibliography}{100}
\bibitem{a_m}J.S. Athreya and G.A. Margulis, Logarithm laws for unipotent flows I, J. Mod. Dyn \textbf{3}(2009), 359-378.
\bibitem{Aus}L. Auslander, Lecture notes on nil-theta functions, Regional conference series in mathematics no.34, American Mathematical Society, 1977.
\bibitem{GNT} R. Garg, A. Nevo and K. Taylor, The lattice point counting problem on the Heisenberg groups, arXiv:1404.6089.
\bibitem{Lang} S. Lang, $\SL(2,\IR)$, Springer-Verlag, 1985.
\bibitem{Rag}M. S. Raghunathan, Discrete subgroups of Lie groups, Springer-Verlag, 1972.
\bibitem{Rand}B. Randol, A group theoretic lattice-counting problem, in Problems in Analysis, Princeton University Press (1970), pp. 291-295.
\bibitem{Rog}C. A. Rogers, Mean values over the space of lattices, Acta Math. vol.94 (1955), pp. 249-287.
\bibitem{Schm} W. Schmidt, A metrical theorem in geometry of numbers. Transactions of the American Mathematical Society (1960), pp. 516-529. 
\bibitem{Sieg}C. L. Siegel, A mean value theorem in the geometry of numbers, Ann. of Math., vol.46 (1945), pp. 340-347.
\end{thebibliography}
\end{document}